\documentclass[11pt]{article}
\usepackage{amssymb,epsfig}
\usepackage{amsmath,amsthm}
\usepackage{amscd}
\usepackage{graphicx}
\usepackage{epstopdf}
\usepackage{color}
\usepackage{multirow,array}

\title {\bf Geometric representation of the weighted harmonic mean of $n$ positive values and potential uses.}

\author{S. Amat\thanks{Departamento de Matem\'atica Aplicada y Estad\'{\i}stica.
 Universidad Polit\'ecnica de Cartagena (Spain).
 e-mail:{\tt sergio.amat@upct.es.} The first four authors have been supported through the Proyecto financiado
por la Comunidad Aut\'onoma de la Regi\'on de Murcia a trav\'es de la convocatoria de Ayudas a proyectos
para el desarrollo de investigaci\'on cient\'{\i}fica y t\'ecnica por grupos competitivos, incluida en el Programa
Regional de Fomento de la Investigaci\'on Cient\'{\i}fica y T\'ecnica (Plan de Actuaci\'on 2018) de la Fundaci\'on
S\'eneca-Agencia de Ciencia y Tecnolog\'{\i}a de la Regi\'on de Murcia 20928/PI/18 and by the Spanish national research project PID2019-108336GB-I00.} \and P. Ortiz\thanks{
 Departamento de Matem\'atica Aplicada y Estad\'{\i}stica.
 Universidad Polit\'ecnica de Cartagena (Spain).
 e-mail:{\tt portiz@navantia.es.}} \and
 J. Ruiz\thanks{
 Departamento de Matem\'atica Aplicada y Estad\'{\i}stica.
 Universidad Polit\'ecnica de Cartagena (Spain).
e-mail:{\tt juan.ruiz@upct.es.}}
 \and
 J.C.Trillo\thanks{
 Departamento de Matem\'atica Aplicada y Estad\'{\i}stica.
 Universidad Polit\'ecnica de Cartagena (Spain).
e-mail:{\tt jc.trillo@upct.es.}} \and
D. F. Ya\~nez \thanks{
 Departamento de Matem\'aticas.
 Universidad de Valencia (Spain).
e-mail:{\tt dionisio.yanez@uv.es.}}
}

\begin{document}

\maketitle

\newtheorem{proposition}{Proposition}
\newtheorem{lemma}{Lemma}
\newtheorem{definition}{Definition}
\newtheorem{theorem}{Theorem}
\newtheorem{corollary}{Corollary}
\newtheorem{remark}{Remark}

\begin{abstract}
This paper is dedicated to the analysis and detailed study of a procedure to generate both the weighted arithmetic and harmonic means of $n$ positive real numbers.
Together with this interpretation, we prove some relevant properties that will allow us to define numerical approximation methods in several dimensions
adapted to discontinuities.
\end{abstract}


\vspace{10pt} {\bf AMS(MOS) subject classifications.} 41A05, 41A10, 65D17.

\section{Introduction}\label{sec1}

Both the arithmetic and the harmonic means of positive numbers appear in different scientific scenarios varying from subdivision schemes and image processing to signal filtering, solution of partial differential equations, statistics, etcetera. The harmonic mean penalizes large
values in the given data, being appropriate, because of this characteristic, for several real world applications. Besides, when the input variables are similar, both means keep that similarity, which
is also convenient in applications as it will be seen later in the paper.

Among the applications of these two means we can mention for instance the following: in the field of numerical solution of hyperbolic conservation laws \cite{Susa,Mar}, for applications involving signal processing
\cite{ADLT,ADLT02,AL04,Trillo_thesis}, in more specific research areas such as image compression and image denoising \cite{ALRT,Den1}, in the fast generation of curves and surfaces by means of subdivision schemes
\cite{ADT,ACRT,KD}.

In \cite{OT3} a nonlinear reconstruction operator called PPH (Piecewise Polynomial Harmonic) was extended to nonuniform grids by using a specific weighted harmonic mean instead of the standard harmonic mean. In this paper our aim is to introduce some necessary ingredients to extend in turn this last reconstruction operator to several dimensions. More specifically speaking, we need to dispose of an appropriate mean in several dimensions which satisfies the required basic properties, the two mentioned above, as the harmonic mean does. We carry out this study accompanied by a geometric representation of the weighted harmonic mean of several values, which helps to quickly and intuitively understand the theoretical results.

Nonlinear means appear as good candidates to define adapted reconstruction methods which mi\-nimize the undesirable effects generated by the presence of a discontinuity in the data. In fact, we will give a list of already existing methods sharing these properties, and in turn we will summarize how to define families of new methods based on the theory and ideas developed through this paper.

The paper is organized as follows: In Section \ref{sec2} we work with the weighted arithmetic and harmonic means of two positive numbers, proving two essential results about these means which will allow us to define
 adapted reconstruction operators in the numerical experiments section. These properties come accompanied with an intuitive graphical interpretation in $2D$ according to a corresponding theoretical result that will be
 also proven. In Section \ref{sec3} a similar path will be followed for the $3D$ case, which involves working with weighted and harmonic means of three positive numbers.  Section \ref{sec4} deals with the general case of considering the weighted arithmetic and harmonic mean of $n$ positive numbers for whatever integer value $n\geq 2.$  In Section \ref{sec5} we outline some applications of these results in order to define adapted reconstructions in several dimensions. Finally, in Section \ref{sec6} we give some conclusions.

\section{About specific results on the weighted harmonic mean of two positive values} \label{sec2}
In this section we present an intuitive graphical interpretation of the weighted arithmetic and harmonic means of two positive values
together with two key results about the weighted harmonic mean that justify their use in several fields of application. Among them we can mention image processing, curve and surface generation,  numerical
 approximation of the solution of hyperbolic conservation laws apart from more traditional uses in statistics and physics. Perhaps the better known problem where the weighted harmonic mean appears is in the computation of the average speed of a vehicle that drives
along a path divided into two parts of different lengths $s_1$ and $s_2$ at constant speed $v_1$ and $v_2$ respectively, that is
\begin{equation*}
v_a=\frac{s_1+s_2}{t_1+t_2}=\frac{s_1+s_2}{\frac{s_1}{v_1}+\frac{s_2}{v_2}}=\frac{1}{w_1\frac{1}{v_1}+w_2\frac{1}{v_2}},
\end{equation*}
with $w_1=\frac{s_1}{s_1+s_2}, w_2=\frac{s_2}{s_1+s_2}.$

 The weighted harmonic mean $H_w$ is given in the following definition.

\begin{definition}\label{M2}
Given $a_1>0,$ $a_2>0$ two positive real numbers and two weights $w_1>0,$ $w_2>0$ with $w_1+w_2=1,$ the weighted harmonic mean of $a_1$ and $a_2$ is defined by
\begin{equation*} \label{eq:SpeedWH}
	H_w(a_1,a_2)= \dfrac{a_1 a_2}{w_1 a_2 + w_2 a_1}.
\end{equation*}
\end{definition}

We now present two particular properties, which have been already used in \cite{OT4} in order to work with a nonlinear reconstruction for nonuniform grids adapted to the potential presence of jump discontinui\-ties on the signal. The first property has to do with the adaptation in case of jump discontinuities, while the second property is related to the order of approximation attained by the nonlinear reconstruction operator, see \cite{OT4} for more details.

\begin{lemma} \label{lemaAcotacionMinimo}
If $a_1 >0 $ and $a_2 > 0,$  the weighted harmonic mean is bounded as follows
	\begin{equation} \label{eq:boundedVnu}
		H_w(a_1,a_2) < \min\left\{ \dfrac{1}{w_1}a_1,\dfrac{1}{w_2}a_2\right\}.
	\end{equation}
\end{lemma}

\begin{lemma}\label{lemaOrdenMedias}
Let $a>0$ a fixed positive real number, and let $a_1\geq a$ and $a_2\geq a.$ If  $\, |a_1-a_2|= O(h)$, then the weighted harmonic mean is also
close to the weighted arithmetic mean $M_w(a_1,a_2)=w_1 a_1+ w_2 a_2,$
	\begin{equation}\label{eq:distMVnu}
		|M_w(a_1, a_2)-H_w(a_1,a_2)|= \dfrac{w_1 w_2}{w_1 a_2+w_2 a_1}(a_1-a_2)^2= O(h^2).
	\end{equation}

\end{lemma}

A way of intuitively check these two properties graphically is by using the following interpretation. Given $a_1,$ $a_2$ two positive numbers and considering $H_w$ the weighted harmonic mean
of these values, we can build the following two parabolas
\begin{equation}\label{eqA:parabolas}
	\begin{array}{lllcl}
			p_1(x) & = & \dfrac{a_1 x_H - \frac{H_w}{2}}{x_H(1-x_H)} x^2 & + & \dfrac{\frac{H_w}{2} - a_1 {x_H}^2}{x_H(1-x_H)} x,\\
			p_2(x) & = & \dfrac{\frac{H_w}{2} + a_2(x_H - 1)}{x_H(x_H-1)} x^2 & - & \dfrac{\frac{H_w}{2} + a_2 (x_H-1)(x_H+1)}{x_H(x_H-1)} x + a_2,
	\end{array}
\end{equation}
\noindent where $x_H$ is defined as the abscissa of the point where both parabolas intersect inside the trapezoid delimited by the four vertices $(0,0),$ $(1,0),$ $(1,a_1),$ $(0,a_2).$ Its value is given by
\begin{equation} \label{eq:xSpeedWH}
	x_H = \dfrac{w_1 a_2}{w_1 a_2 + w_2 a_1}.
\end{equation}

\begin{remark}
Geometrically,  one can build the parabolas $p_1(x)$ and $p_2(x)$ as the unique polynomials of degree less or equal to $2$ such that
they interpolate the points $\{(0,0), (\frac{1}{2},(\frac{1}{4}+\frac{1}{8 w_1})a_1+(\frac{1}{4}-\frac{1}{8 w_2})a_2), (1,a_1)\},$
and $\{(0,a_2), (\frac{1}{2},(\frac{1}{4}-\frac{1}{8 w_1})a_1+(\frac{1}{4}+\frac{1}{8 w_2})a_2), (1,0)\}$ respectively.
\end{remark}

In Figure \ref{fig1-Par2} upper-left we can see the representation of the trapezoid with the two parabolas intersecting at a point with abscissa $x_H,$ for similar values of $a_1$ and $a_2$ and for a value of the
weights $w_1=\frac{7}{10},$ $w_2=\frac{3}{10}.$
In this case, it is appreciated a similar value of the weighted harmonic and arithmetic means. This particular situation relates with Lemma \ref{lemaOrdenMedias}.
In Figure  \ref{fig1-Par2} bottom-left we can see the case for quite different values of $a_1$ and $a_2.$  Now, it can be  observed that the weighted harmonic mean remains much closer to the minimum
value between $a_1$ and $a_2$ than the weighted arithmetic mean. This situation has a close relation with Lemma \ref{lemaAcotacionMinimo}.
In Figure \ref{fig1-Par2} upper-right  and bottom-right we consider the case of having equal weights $w_1=w_2=\frac{1}{2},$ which gives rise to the usual arithmetic and harmonic means.
The observations are the same as in the weighted case, although it is interesting to notice that the parabolas degenerate in the two diagonals of the trapezoid.

\begin{figure}[!ht]
\centerline{\includegraphics[width=7cm]{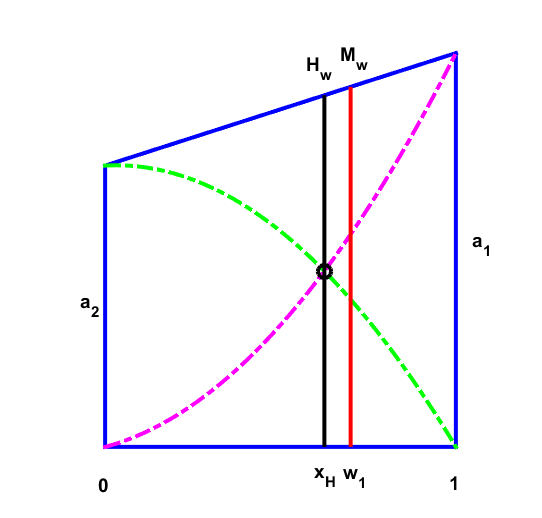}
\includegraphics[width=7cm]{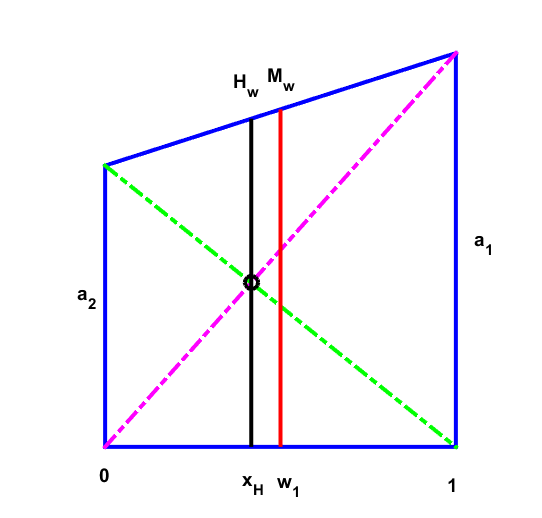}}
\centerline{\includegraphics[width=7cm]{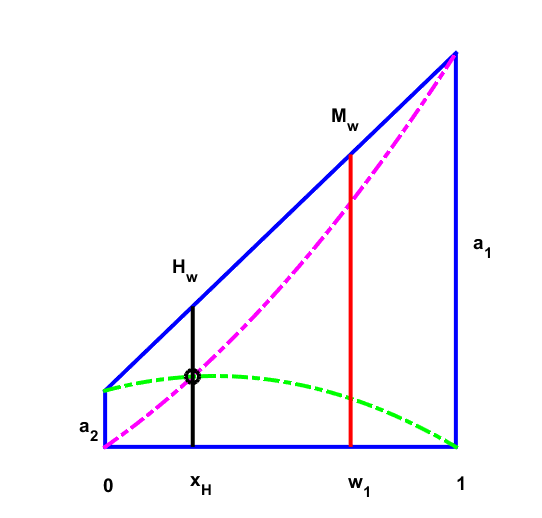}
\includegraphics[width=7cm]{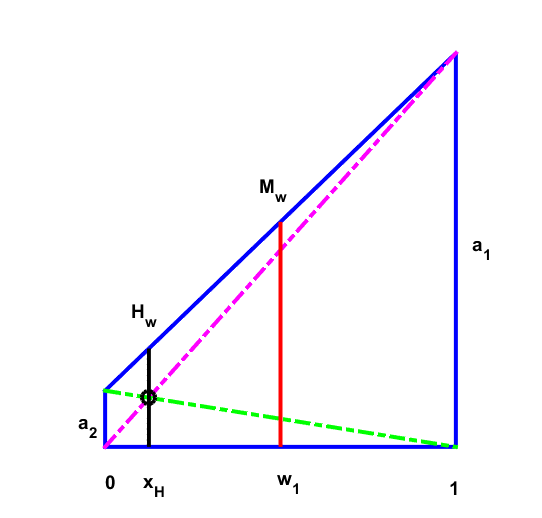}}
\caption{Representation of the weighted harmonic and arithmetic means. Upper-left: $w_1=0.7,$ $w_2=0.3,$  $a_1=14,$ $a_2=10.$  Upper-right: $w_1=0.5,$ $w_2=0.5,$ $a_1=14,$ $a_2=10.$
Bottom-left:  $w_1=0.7,$ $w_2=0.3,$ $a_1=14,$ $a_2=2.$  Bottom-right: $w_1=0.5,$ $w_2=0.5,$  $a_1=14,$ $a_2=2.$
In black the weighted harmonic mean, in red the weighted arithmetic mean, in dashed magenta line the parabola $p_1(x)$ and in dashed green line the parabola $p_2(x).$
} \label{fig1-Par2}
\end{figure}
There are infinitely many ways of defining two parabolas which degenerate in the two diagonals for $w_1=w_2=\frac{1}{2},$ and intersect at the abscissa $x_H$ where $a_2 + (a_1 - a_2) x_H = H_w.$
In fact, for each ordinate of the type $y_H=f(w_1, w_2) a_1 x_H$
with $f(w_2,w_1)=1$ for $w_1=w_2=\frac{1}{2},$ the parabolas interpolating the points $\{(0,0), (x_H,y_H), (1,a_1)\}$ and
$\{(0,a_2), (x_H,y_H), (1,0)\}$ satisfy both requirements. In particular, we remark three particular cases because of their symmetry or simplicity.\\
\\
\textbf{Case 1: $f(w_1, w_2)=\dfrac{w_2}{w_1}$}.\\
In this case we have
\begin{equation}
	y_H = \dfrac{w_2}{w_1} a_1 x_H = \dfrac{w_2 a_2 a_1}{w_1 a_2 + w_2 a_1},
\end{equation}
and the parabolas take the form
\begin{equation}\label{eqA:parabolas01}
	\begin{array}{lcl}
			p_1(x) & = & \dfrac{a_1}{1-x_H}\left[ \left(1-\dfrac{w_2}{w_1}\right)x^2 - \left(x_H-\dfrac{w_2}{w_1}\right)x\right],\\
			p_2(x) & = & a_2-a_2 x.
	\end{array}
\end{equation}
\textbf{Case 2: $f(w_1, w_2)=1$}.\\
In this case we get
\begin{equation}
	y_H = a_1 x_H,
\end{equation}
and the parabolas are given by
\begin{equation}\label{eqA:parabolas02}
	\begin{array}{lcl}
			p_1(x) & = & a_1 x,\\
			p_2(x) & = & \left(\dfrac{a_1}{x_H-1} + \dfrac{a_2}{x_H} \right) x^2 - \left(\dfrac{a_1}{x_H-1} + \dfrac{a_2(x_H+1)}{x_H} \right) x + a_2.
	\end{array}
\end{equation}
\textbf{Case 3. $f(w_1, w_2)= \dfrac{1}{2 w_1}$}.\\
In this case
\begin{equation}
	y_H = \dfrac{1}{2 w_1}a_1 x_H,
\end{equation}
and the parabolas are given in (\ref{eqA:parabolas}).\\
Notice that in the first two cases one of the parabolas remains always equal to one of the diagonals of the trapezoid for all values of
$w_1.$ Since the first and second cases are symmetrical, we will consider only the first and third cases from now on.
In the next section, we will present the geometrical extension of the given results to the three variables case. The proofs will be omitted because they appear later
in the general n-dimensional case.
\section{Geometrical interpretation of the weighted harmonic mean of three positive values} \label{sec3}
In this section we give the corresponding results about the weighted harmonic mean for the case of dealing with three positive values. These results can be generalized to
$n$ values with $n$ a positive integer number, and we will address this situation in the next section, where we will include the proofs.

\begin{definition}\label{M3}
Given $a_1>0,$ $a_2>0,$ $a_3>0$  three positive real numbers and the weights $w_1>0,$ $w_2>0,$ $w_3>0$ with $w_1+w_2+w_3=1,$ their weighted harmonic mean is defined by
\begin{equation*} \label{eq:HW3}
	H_w(a_1,a_2,a_3)= \dfrac{a_1 a_2 a_3}{w_1 a_2 a_3 + w_2 a_1 a_3 + w_3 a_1 a_2}.
\end{equation*}
\end{definition}

\begin{lemma} \label{lemaAcotacionMinimo3}
If $a_1 > 0,$ $a_2 > 0,$ $a_3 >0$ the weighted harmonic mean is bounded as follows
	\begin{equation} \label{eq:boundedVnu3}
		H_w(a_1,a_2,a_3) < \min\left\{ \dfrac{1}{w_1}a_1,\dfrac{1}{w_2}a_2,\dfrac{1}{w_2}a_2\right\}.
	\end{equation}
\end{lemma}

\begin{lemma}\label{lemaOrdenMedias3}
Let $a>0$ a fixed positive real number, and let $a_1\geq a,$ $a_2\geq a,$ $a_3\geq a.$ If  $|a_1-a_2|= O(h),$ $|a_1-a_3|= O(h),$  then the weighted harmonic mean is also
close to the weighted arithmetic mean $M_w(a_1,a_2,a_3)=w_1 a_1+ w_2 a_2+w_3 a_3,$
	\begin{eqnarray}\label{eq:distMVnu3} \notag
		|M_w(a_1, a_2,a_3)-H_w(a_1,a_2,a_3)|&=& \dfrac{w_1 w_2(a_1-a_2)^2 a_3+w_1 w_3(a_1-a_3)^2 a_2+w_2 w_3(a_2-a_3)^2 a_1 }{w_1 a_2 a_3 + w_2 a_1 a_3 + w_3 a_1 a_2}\\
&=& O(h^2).
	\end{eqnarray}

\end{lemma}

The following two theorems are dedicated to write in a formal way the geometrical interpretation of the weighted harmonic mean, generalizing the expressions for the two variables case given
in (\ref{eqA:parabolas}), and (\ref{eqA:parabolas01}). The case of expressions (\ref{eqA:parabolas02}) could be treated in a similar way, and we will not consider it, since it is a symmetrical
version of case (\ref{eqA:parabolas01}).
Let us first introduce the following notations for the vertices of a straight prism with triangular base

\begin{eqnarray*}
B_1&=&(1,0,0), \quad B_2=(0,1,0), \quad B_3=(0,0,0),\\
P_1&=&(1,0,a_1), \quad P_2=(0,1,a_2), \quad P_3=(0,0,a_3),\\
\end{eqnarray*}
where $B_i, i=1,2,3,$ stand for the vertices of the base and the corresponding $P_i$ for the vertices located at the heights of the prism through the points $B_i$
satisfying that the length of the segment between $P_i$ and $B_i$ is $a_i.$ We will also use the barycenter of the points $B_i$
\begin{equation*}
GM_w:=w_1 B_1+w_2 B_2 + w_3B_3.
\end{equation*}

The first theorem amounts to the generalization of the expressions in (\ref{eqA:parabolas01}) and can be written as follows.
\begin{theorem} \label{teo:3Dcase1}
Let us consider the plane $\Pi$ which passes through the points $P_1,$ $P_2$ and $P_3$ given by the equation
\begin{equation}\label{eq3A:Plano_P1P2P3_1}
	\Pi \equiv  x_{1}(a_{3}-a_{1}) + x_{2}(a_{3}-a_{2}) + x_{3}-a_{3}=0.
\end{equation}
Let us also consider the plane $V_{3}$ which passes through the points $B_1,$ $B_2$ and $P_3$ given by the equation
\begin{equation} \label{eqPlane1-3D}
V_3 \equiv x_1+x_2+\frac{x_3}{a_3}=1,
\end{equation}
and the two paraboloids $V_1$ and $V_2$ given by the equations
\begin{eqnarray} \label{eqParab1-3D}
V_1 &\equiv&	x_{3}=b_{1}{x_{1}}^{2} + (a_{1} - b_{1})x_{1}, \quad \textrm{which passes through} \ P_1, B_2, B_3,\\ \notag
V_2 &\equiv&    x_{3}=b_{2}{x_{2}}^{2} + (a_{2} - b_{2})x_{2}, \quad  \textrm{which passes through} \ B_1, P_2, B_3,
\end{eqnarray}
where the coefficients $b_i$ are given by
\begin{equation}\label{eqnHwT1:Coef_bi_xH3D}
	b_{i}=\dfrac{H_{w}}{\bar{x}_{i}(\bar{x}_{i}-1)} (w_{3}-w_{i}), \qquad i=1,2.
\end{equation}\\
Then, the system of equations formed by (\ref{eqPlane1-3D}) and (\ref{eqParab1-3D}) has a unique solution $(\bar{x}_{1},\bar{x}_2,\bar{x}_{3})$ given by
\begin{equation}\label{eqnHwT1:Punto_xH3D}
\bar{x}_{1}=w_{1} \dfrac{H_{w}}{a_{1}}, \quad \bar{x}_{2}=w_{2} \dfrac{H_{w}}{a_{2}}, \quad	\bar{x}_{3}=w_{3}H_{w}.
\end{equation}
Moreover, the height of the prism through the point $(\bar{x}_1,\bar{x}_2,0)$ coincides with the weighted harmonic mean $H_w$ of $a_1,$ $a_2,$ $a_3$ and the height of the prism through the barycenter of the
triangular base $GM_w=w_1B_1+w_2B_2+w_3B_3$ coincides with the weighted arithmetic mean.
\end{theorem}

The second theorem deals with the generalization of expressions (\ref{eqA:parabolas}).

\begin{theorem} \label{teo:3Dcase3}
Let us consider the plane $\Pi$ which passes through the points $P_1,$ $P_2$ and $P_3$ given by the equation
\begin{equation}\label{eq3A:Plano_P1P2P3_3}
	\Pi \equiv  x_{1}(a_{3}-a_{1}) + x_{2}(a_{3}-a_{2}) + x_{3}-a_{3}=0.
\end{equation}
Let us also consider the paraboloid $V_3^*$ passing through $B_1,$ $B_2,$ $P_3$ given by
\begin{equation}\label{eqnHwT2:Parabol_Pin3D}
	\begin{aligned}
		x_{3} &= a_3 + (c_{1}x_{1}(x_1-1)-a_3 x_1) + (c_{2}x_{2}(x_2-1)-a_3 x_2), \\
	\end{aligned}
\end{equation}\\
where the coefficients $c_i$ are given by
\begin{equation}\label{eqnHwT1:Coef_ci_xH3D}
	c_{i}=\dfrac{\frac{H_{w}}{3}+(\bar{x}_{1}+\bar{x}_2-1)a_3}{2 \bar{x}_i(\bar{x}_{i}-1)}, \qquad i=1,2,
\end{equation}
and the two paraboloids $V_1$ and $V_2$ given by the equations
\begin{eqnarray} \label{eqParab3-3D}
V_1^* &\equiv&	x_{3}=b_{1}{x_{1}}^{2} + (a_{1} - b_{1})x_{1}, \quad \textrm{which passes through} \ P_1, B_2, B_3,\\ \notag
V_2^* &\equiv&    x_{3}=b_{2}{x_{2}}^{2} + (a_{2} - b_{2})x_{2}, \quad  \textrm{which passes through} \ B_1, P_2, B_3,
\end{eqnarray}
where the coefficients $b_i$ are given by
\begin{equation}\label{eqnHwT1:Coef_bi3_xH3D}
	b_{i}=\dfrac{H_{w}}{\bar{x}_{i}(\bar{x}_{i}-1)} (\frac{1}{3}-w_{i}), \qquad i=1,2.
\end{equation}\\
Then, the system of equations formed by (\ref{eqnHwT2:Parabol_Pin3D}) and (\ref{eqParab3-3D}) has a unique solution $(\bar{x}_{1},\bar{x}_2,\bar{x}_{3})$ given by
\begin{equation}\label{eqnHwT2:Punto_xH3D}
\bar{x}_{1}=w_{1} \dfrac{H_{w}}{a_{1}}, \quad \bar{x}_{2}=w_{2} \dfrac{H_{w}}{a_{2}}, \quad	\bar{x}_{3}=\frac{H_{w}}{3}.
\end{equation}
Moreover, the height of the prism through the point $(\bar{x}_1,\bar{x}_2,0)$ coincides with the weighted harmonic mean $H_w$ of $a_1,$ $a_2,$ $a_3$ and the height of the prism through the barycenter of the
triangular base $GM_w=w_1B_1+w_2B_2+w_3B_3$ coincides with the weighted arithmetic mean.
\end{theorem}

\indent In Figures \ref{fig1:3D} and  \ref{fig2:3D} we represent the situation given in Theorem \ref{teo:3Dcase3}, being the situation of  Theorem \ref{teo:3Dcase1} similar.
In Figure \ref{fig1:3D}, in the left part, we show the paraboloids built with the values $a_1=3,$ $a_2=4,$ $a_3=6,$ with the weights $w_1=0.2,$ $w_2=0.2,$ $w_3=0.6,$ and in the right part, the planes
obtained for the case of dealing with equal weights $w_1=w_2=w_3=\frac{1}{3}.$ These plots correspond to the situation considered in Theorem \ref{teo:3Dcase3}. We observe how the paraboloids degenerate in planes
generalizing the case of the non-weighted harmonic mean.\\
In Figure \ref{fig2:3D}, we show the intersection of the three paraboloids for the same values and weights. It is interesting to compare the representation of the weighted harmonic mean $H_w,$ which coincides with
the height of the prism through the point $GH_w$ (orthogonal projection onto the base of the intersection point of the three paraboloids considered in Theorem \ref{teo:3Dcase3}), with the representation of the weighted arithmetic mean $M_w,$ which amounts to the height of the prism through the barycenter $GM_w$ of the vertices of the triangular base affected by the corresponding weights.

\begin{figure}[!ht]
\centerline{\includegraphics[width=7cm]{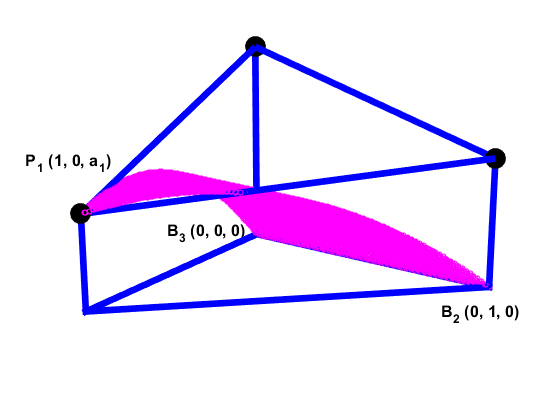}
\includegraphics[width=7cm]{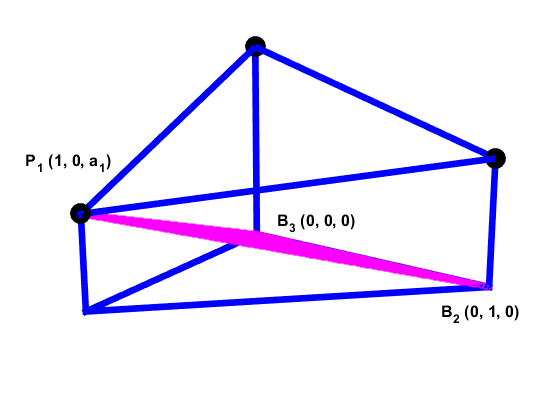}}
\centerline{\includegraphics[width=7cm]{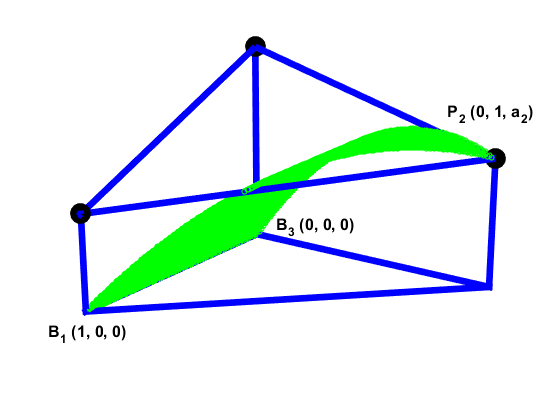}
\includegraphics[width=7cm]{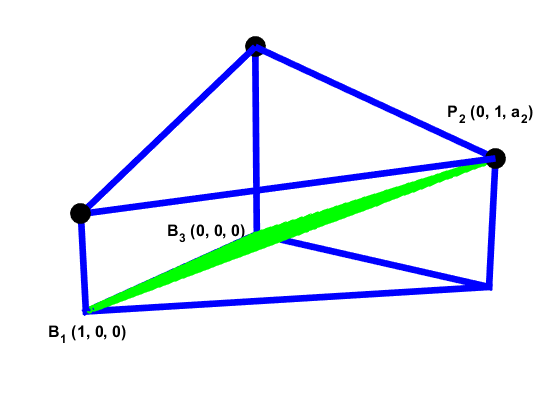}}
\centerline{\includegraphics[width=7cm]{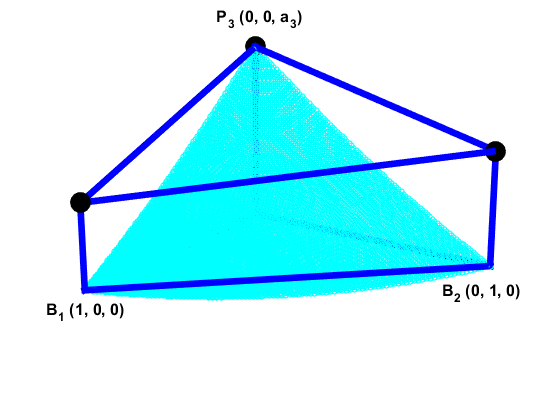}
\includegraphics[width=7cm]{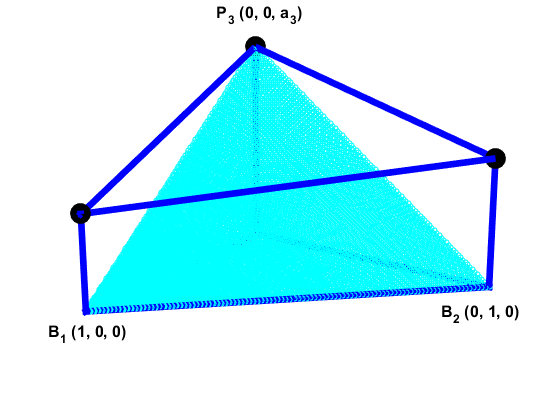}}
\caption{Representation of the three paraboloids considered in Theorem \ref{teo:3Dcase3} for the representation of the harmonic mean of the values
$a_1=3,$ $a_2=4,$ $a_3=6.$ Weights $w_1=0.2,$ $w_2=0.2,$ $w_3=0.6$ to the left and $w_1=w_2=w_3=\frac{1}{3}$ to the right. Upper: $V_1^*.$
Medium:  $V_2^*.$ Bottom: $V_3^*.$
} \label{fig1:3D}
\end{figure}

\begin{figure}[!ht]
\centerline{\includegraphics[width=7cm]{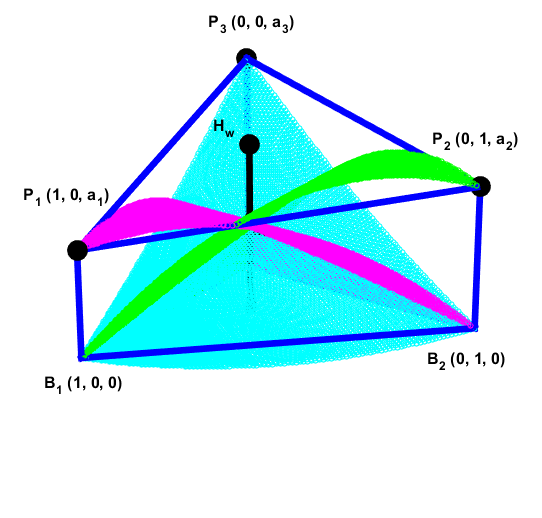}
\includegraphics[width=7cm]{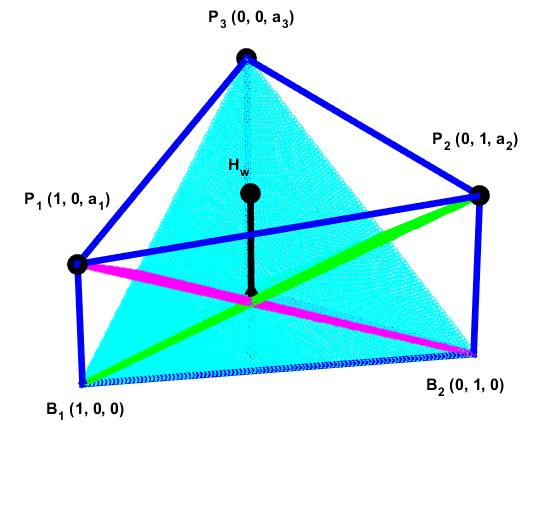}}
\centerline{\includegraphics[width=7cm]{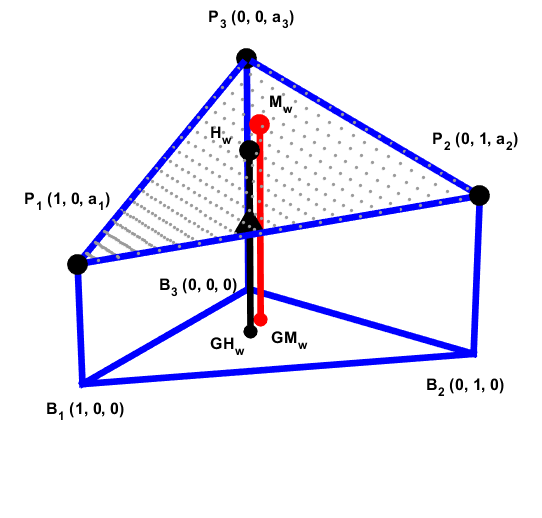}
\includegraphics[width=7cm]{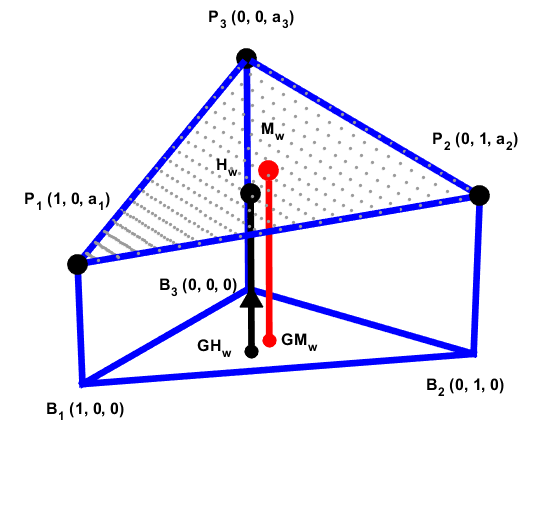}}
\caption{Representation of weighted harmonic mean of three positive values $a_1=3,$ $a_2=4,$ $a_3=6$ as the height of the prism through the intersection point of the three paraboloids considered in Theorem \ref{teo:3Dcase3}. Comparison with the weighted arithmetic mean.  for the representation of the harmonic mean of the values.
Weights $w_1=0.2,$ $w_2=0.2,$ $w_3=0.6$ to the left and $w_1=w_2=w_3=\frac{1}{3}$ to the right. Upper: Intersection of the three paraboloids.
Bottom: Comparison between the weighted harmonic mean and the weighted arithmetic mean.
} \label{fig2:3D}
\end{figure}

\section{Results on the weighted harmonic mean of $n$ values} \label{sec4}

First, we introduce the definition of weighted harmonic mean $H_w$ that we are going to be using.

\begin{definition}\label{Mn}
Given $a_i>0, i=1,\ldots,n$  $n$ positive real numbers and the weights $w_i>0, i=1,\ldots,n$ with $\sum_{i=1}^{n}w_i=1,$ the weighted harmonic mean is defined by
\begin{equation*}\label{eqnHw:MediaArmonica}
	H_{w}(a_1,\ldots,a_n) = \dfrac{1}{\sum\limits_{i=1}^{n}\dfrac{w_{i}}{a_{i}}}= \dfrac{\prod\limits_{k=1}^{n}a_{k}}{\sum\limits_{i=1}^{n}w_{i}\prod\limits_{\substack{k=1 \\k \neq i}}^{n}{a_{k}}},
\end{equation*}
and the weighted arithmetic mean is defined by
\begin{equation*}
M_w(a_1,\ldots,a_n)=\sum_{i=1}^n w_i a_i.
\end{equation*}
\end{definition}

We now give the main two results which are crucial in applications in numerical analysis, such as we will show in the section devoted to practical cases.
The first lemma has to do with the property of boundedness of the mean by the minimum of its arguments and it is used to define adaptative methods.
\begin{lemma} \label{minimon}
Let $a_i>0, i=1,\ldots,n$ be  $n$ positive real numbers and $w_i>0, i=1,\ldots,n$ the correspon\-ding weights with $\sum_{i=1}^{n}w_i=1.$  Then, the weighted harmonic mean $H_w$ is bounded as follows
\begin{equation*}
	H_{w} < \dfrac{a_{i_{0}}}{w_{i_{0}}} \leq \dfrac{a_{i}}{w_{i}}, \quad i=1,\ldots, n,
\end{equation*}
where $\dfrac{a_{i_{0}}}{w_{i_{0}}} = \min \{ \dfrac{a_{1}}{w_{1}}, \cdots, \dfrac{a_{n}}{w_{n}} \}.$
\end{lemma}

\begin{proof}

\begin{equation*}
	H_{w}= \dfrac{\prod\limits_{k=1}^{n}a_{k}}{\sum\limits_{j=1}^{n}{w_{j}\prod\limits_{\substack{k=1 \\k \neq j}}^{n}{a_{k}}}} = \dfrac{a_{i_0}}{w_{i_0}} \, \dfrac{\prod\limits_{\substack{k=1 \\k \neq i_0}}^{n}{a_{k}}}{\sum\limits_{j=1}^{n}\dfrac{w_{j}}{w_{i_0}}\prod\limits_{\substack{k=1 \\k \neq j}}^{n}{a_{k}}} < \dfrac{a_{i_{0}}}{w_{i_{0}}} \leq \dfrac{a_{i}}{w_{i}}, \quad i=1,\ldots, n.
\end{equation*}
\end{proof}

The second lemma deals with how close remains the weighted harmonic mean to the weighted arithmetic mean when the arguments are also close among them. This property is essential to define nonlinear methods which
preserve the order of approximation of their linear counterparts from which they are derived. We will also show this relation in the section dedicated to the practical examples.
\begin{lemma} \label{ordenn}
Let $a_i>0, i=1,\ldots,n$ be  $n$ positive real numbers and $w_i>0, i=1,\ldots,n$ the corres\-ponding weights with $\sum_{i=1}^{n}w_i=1.$
If $\;a_{i} = O(1), \quad \forall i=1, \cdots, n,\quad$ and $\quad\left|a_{1}-a_{i}\right| = O(h), \quad \forall i=2, \cdots, n, \quad$ then, the weighted harmonic mean $H_w$ and the weighted arithmetic
mean $M_w:=\sum_{i=1}^{n}w_ia_i$ satisfy
\begin{equation*}
\left|M_{w} - H_{w} \right|=O(h^{2}).
\end{equation*}
\end{lemma}

\begin{proof}

Using the expressions of $H_w$ and $M_w$ we have
\begin{equation} \label{MwHw}
|M_w - H_w|=	\left| \sum\limits_{i=1}^{n}w_{i}a_{i} - \dfrac{\prod\limits_{k=1}^{n}a_{k}}{\sum\limits_{j=1}^{n}w_{j}\prod\limits_{\substack{k=1 \\k \neq j}}^{n}{a_{k}}} \right|   =    \left|\dfrac{\sum\limits_{i=1}^{n} w_{i} a_{i}\sum\limits_{j=1}^{n}w_{j}\prod\limits_{\substack{k=1 \\k \neq j}}^{n}{a_{k}}- \prod\limits_{k=1}^{n}a_{k}}{\sum\limits_{j=1}^{n}w_{j}\prod\limits_{\substack{k=1 \\k \neq j}}^{n}{a_{k}}}\right|.
\end{equation}
Now, paying attention to the fact that given two indices $\;i_{0}, j_{0}\;$ such that $\;1 \leq i_{0}< j_{0} \leq n\;$ we have
\begin{equation} \label{i0j0}
	w_{i_{0}}a_{i_{0}}w_{j_{0}}\prod\limits_{\substack{k=1 \\k \neq j_{0}}}^{n}{a_{k}} =  w_{i_{0}}a_{i_{0}}^2 w_{j_{0}}\prod\limits_{\substack{k=1 \\k \neq i_{0},j_{0}}}^{n}{a_{k}},
\end{equation}
\begin{equation} \label{j0i0}
	w_{j_{0}}a_{j_{0}}w_{i_{0}}\prod\limits_{\substack{k=1 \\k \neq i_{0}}}^{n}{a_{k}} =  w_{i_{0}}a_{j_{0}}^2 w_{j_{0}}\prod\limits_{\substack{k=1 \\k \neq i_{0},j_{0}}}^{n}{a_{k}},
\end{equation}
and just by summing up both terms in (\ref{i0j0}) and (\ref{j0i0}) we get
\begin{equation} \label{i0j0&j0i0}
w_{i_{0}}a_{i_{0}}w_{j_{0}}\prod\limits_{\substack{k=1 \\k \neq j_{0}}}^{n}{a_{k}}+w_{j_{0}}a_{j_{0}}w_{i_{0}}\prod\limits_{\substack{k=1 \\k \neq i_{0}}}^{n}{a_{k}}
= w_{i_{0}}w_{j_{0}}\prod\limits_{\substack{k=1 \\k \neq i_{0}, j_{0}}}^{n}{a_{k}}(a_{i_{0}}^2 + a_{j_{0}}^2).
\end{equation}
For the case $i_0=j_0$ we get
\begin{equation}\label{i0i0}
w_{i_{0}}a_{i_{0}}w_{j_{0}}\prod\limits_{\substack{k=1 \\k \neq j_{0}}}^{n}{a_{k}} =  w_{i_{0}}^2 \prod\limits_{k=1}^{n}{a_{k}}.
\end{equation}
Using the simplifications in (\ref{i0j0&j0i0}) and (\ref{i0i0}) we can rewrite (\ref{MwHw}) as
\begin{equation*}
|M_w - H_w|	=   \left|\dfrac{\sum\limits_{i=1}^{n}{w_{i}}^2\prod\limits_{k=1}^{n}a_{k}+\sum\limits_{\substack{i,j=1 \\i < j}}^{n}w_{i} w_{j} (a_{i}^2+a_{j}^2) \prod\limits_{\substack{k=1 \\k \neq i,j}}^{n}{a_{k}}}{\sum\limits_{j=1}^{n}w_{j}\prod\limits_{\substack{k=1 \\k \neq j}}^{n}{a_{k}}}\right|    =    \left|\dfrac{\sum\limits_{\substack{i,j=1 \\i < j}}^{n}w_{i} w_{j} (a_{i}-a_{j})^2 \prod\limits_{\substack{k=1 \\k \neq i,j}}^{n}{a_{k}}}{\sum\limits_{j=1}^{n}w_{j}\prod\limits_{\substack{k=1 \\k \neq j}}^{n}{a_{k}}}\right|  =   O(h^{2}),
\end{equation*}
since by the triangular inequality we have that $\; \left|a_{i} - a_{j}\right| \leq \left|a_{i} - a_{1}\right| + \left|a_{1} - a_{j}\right| = O(h).$
\end{proof}

We introduce the following notation for the vertices of a prism in $\mathbb{R}^n.$
\begin{equation*}\label{eqnAH:Prisma}
	\begin{cases}
			B_{i} \equiv (0, \cdots, 0,\underset{i}{1},0,\cdots,0) \quad i=1,\cdots,n-1,\\
			B_{n}= (0, \cdots, 0),\\
			P_{i} \equiv (0, \cdots, 0,\underset{i}{1},0,\cdots,0, a_{i}) \quad i=1,\cdots,n-1,\\
			P_{n}= (0, \cdots, 0, a_{n}),
	\end{cases}
\end{equation*}
where the points $B_i$ represent the vertices which lay on the base of the prism and the vertices $P_i$ are nothing more than the points located at the maximum height of the prism at the
corresponding points $B_i$ in the base and in the parallel direction to the $x_n$ axis. \\

We are now ready to give the following two theorems for the weighted harmonic mean, which generalize the geometrical representations using prisms.

\begin{theorem} \label{teo:NDcase1}
Let us consider the hyperplane $\Pi$ which passes through the points $P_i, i=1,\ldots,n,$ $n\geq 2,$ given by the equation
\begin{equation}\label{eq:Plano_Pi_1}
	\Pi\equiv \; x_n=a_{n} + \sum_{i=1}^{n}{x_{i}(a_{i}-a_{n})}.
\end{equation}
Let us also consider the hyperplane $V_{n}$ which passes through the points $B_i, i=1,\ldots,n-1$ and $P_n$ given by the equation
\begin{equation} \label{eqPlane1-ND}
V_n \equiv  \; \sum_{i=1}^{n-1}{x_{i}} + \dfrac{x_{n}}{a_{n}}=1.
\end{equation}
and the paraboloids $V_i, i=1,\ldots,n-1$ given by the equations
\begin{eqnarray} \label{eqParab1-ND}
V_i &\equiv&	x_{n}=b_{i}{x_{i}}^{2} + (a_{i} - b_{i})x_{i},
\end{eqnarray}
which pass through $B_1,\ldots,B_{i-1},P_i,B_{i+1},\ldots, B_n$ respectively, where the coefficients $b_i$ are given by
\begin{equation}\label{eqnHwT3:Coef_bi_xH3D}
	b_{i}=\dfrac{H_{w}}{\bar{x}_{i}(\bar{x}_{i}-1)} (w_{n}-w_{i}), \qquad i=1,\ldots,n-1.
\end{equation}\\
Then, the system of equations formed by (\ref{eqPlane1-ND}) and (\ref{eqParab1-ND}) has a unique solution $(\bar{x}_{1},\ldots,\bar{x}_{n})$ given by
\begin{equation}\label{eqnHwT3:Punto_xH3D}
\bar{x}_{i}=w_{i} \dfrac{H_{w}}{a_{i}}, \quad i=1,\ldots,n-1, \quad	\bar{x}_{n}=w_{n}H_{w}.
\end{equation}
Moreover, the following two affirmations are true:
\begin{itemize}
\item[a)] The height of the prism through the point $(\bar{x}_1,\ldots,\bar{x}_{n-1},0)$ coincides with the weighted harmonic mean $H_w$ of $a_i,i=1,\ldots,n,$ that is, the point
 $(\bar{x}_1,\ldots,\bar{x}_{n-1},H_w)$ belongs to the hyperplane $\Pi.$
\item[b)] The height of the prism through the barycenter of the
base $GM_w=\sum\limits_{i=1}^{n}w_iB_i$ coincides with the weighted arithmetic mean.
\end{itemize}
\end{theorem}

\begin{proof}

It is immediate to check that the proposed solution satisfies (\ref{eqPlane1-ND}) and (\ref{eqParab1-ND}).
Let us prove that the solution is unique. By reductio ad  absurdum, let us suppose that there exists another solution
$x^{\prime}=(x^{\prime}_{1}, \cdots, x^{\prime}_{n})$ with $x^{\prime}\neq \bar{x}.$ Then, denoting $z_{i}= \bar{x}_{i}-x^{\prime}_{i}, i=1,\cdots,n,$ the system of equations formed by
(\ref{eqPlane1-ND}) and (\ref{eqParab1-ND}) can be easily transformed into	
\begin{subequations}\label{eqnHwT1:T1_z}
	\begin{align}
			\sum\limits_{i=1}^{n-1}{z_{i}} + \dfrac{z_{n}}{a_{n}} &=0, \label{eqnHwT1:T1_z_1} \\
			z_{n}-b_{i}z_{i}(\bar{x}_{i}+x^{\prime}_{i}) - (a_{i} - b_{i})z_{i} &= 0,  \quad i=1,\cdots,n-1, \label{eqnHwT1:T1_z_2}
	\end{align}
\end{subequations}
what amounts to a homogeneous linear system of $n$ equations with $n$ unknowns. If we show that this system has only the trivial solution $z=0,$ then we would have proven that
$\bar{x}=x^{\prime},$ what is a contradiction with the starting supposition. Therefore, $\bar{x}$ would be the unique solution. Let us then prove that system (\ref{eqnHwT1:T1_z}) has
$z=0$ as the unique solution. Again by reductio ad absurdum, let us suppose that the system has infinite solutions, that is,
$z= \sum\limits_{k=1}^{s}{\lambda_{k}v^{k}}$, where $s=n-r,$ being $r$ the rank of the coefficient matrix of the linear system, $\lambda_{k} \in \mathbb{R},$ and $v^k, \ k=1,\ldots,s,$ represent
a base of the kernel of the associated linear map. Let us consider the univariate set of solutions $z=\lambda_1 v^1,$ $\lambda_1 \in \mathbb{R}.$ By the sake of simplicity, we will drop the superindex and we will
write $z=\lambda v.$ Thus, we obtain
\begin{equation*}
x^{\prime}=\bar{x}-z= \bar{x}- \lambda v,
\end{equation*}
whose coordinates are given by\\
\begin{equation} \label{xprima1}
x^{\prime}_{i}=\bar{x}_{i}-z_{i}= \bar{x}_{i}-\lambda v_{i}.
\end{equation}
Plugging (\ref{xprima1}) into (\ref{eqnHwT1:T1_z_2}) we get\\
\begin{equation} \label{eq:lambda1}
\lambda v_{n}=b_{i} \lambda v_{i}(2 \bar{x}_{i}-\lambda v_{i}) + (a_{i} - b_{i})\lambda v_{i}, \quad \forall i=1,\cdots,n-1,
\end{equation}
and simplifying expression (\ref{eq:lambda1}) we obtain
\begin{equation}\label{eq:lambda2}
	-\lambda b_{i} v_{i}^{2} + 2 b_{i} v_{i} \bar{x}_{i}+(a_{i} - b_{i})v_{i} - v_{n}=0, \quad \forall i=1,\cdots,n-1, \quad \lambda \neq 0.
\end{equation}
Now, particularizing expression (\ref{eq:lambda2}) for two different values of $\lambda,$  $\lambda_{1} \neq \lambda_{2},$ and subtracting both expressions, we reach to
\begin{equation}\label{eq:lambda3}
	(\lambda_{1} - \lambda_{2}) b_{i} v_{i}^{2} = 0, \quad \forall i=1, \cdots, n-1.
\end{equation}
We are going to prove now that there exists $i_0 \in \left\{1, \cdots, n-1 \right\}$ such that $ b_{i_0} \neq 0 \; \text{and} \; v_{i_0} \neq 0,$ and therefore, from (\ref{eq:lambda3}), this would imply that
$\; \lambda_{1} = \lambda_{2}$ what is a contradiction. Thus, $z=0$ would be the unique solution of the homogeneous linear system and $\bar{x}$ would be the unique solution of
 the system given by (\ref{eqPlane1-ND}) and (\ref{eqParab1-ND}).

Since $v\neq 0,$ $ \exists v_i \neq 0 \ \text{for some} \ i \in \left\{1, \cdots, n-1 \right\}.$ Otherwise, if $v_i = 0, \ \forall i \in \left\{1, \cdots, n-1 \right\}$,
from (\ref{eqnHwT1:T1_z_1}) we get $v_{n}=-a_{n}\sum\limits_{i=1}^{n-1}{v_{i}} = 0\;$ and $\;v=0,$ what is not possible.
Let us denote $I$ the set of indices for which $v_{i}\neq0$. If we suppose that $b_{i}=0, \ \forall i \in I,$ then from (\ref{eq:lambda2}) we get
$\;a_{i}v_{i}-v_{n}=0.$ Thus,  $v_{i}= \dfrac{v_{n}}{a_{i}}, \ \forall i \in I.$ Also, from (\ref{eqnHwT1:T1_z_1})
\begin{equation} \label{eq:lambda4}
\sum\limits_{i\in I}{\lambda v_{i}} + \lambda \dfrac{v_{n}}{a_{n}}=0.
\end{equation}
Now, using in (\ref{eq:lambda4}) the fact that $v_{i}= \dfrac{v_{n}}{a_{i}}, \ \forall i \in I,$ we get $v_n=0,$ and in turn, $v=0,$ what gives a contradiction which comes from the supposition
$b_{i}=0, \ \forall i \in I.$ Therefore, $\exists i_0 \in I,$ such that $b_{i_0}\neq 0.$   \\

In order to prove now point a) of the theorem, we consider the straight line parallel to the $x_{n}$ axis passing through  $(\bar{x}_1,\ldots,\bar{x}_{n-1},0)$, that is
\begin{equation}\label{eqnAw:Recta_r}
	r_{w}\equiv
	\begin{cases}
			x_{1}= \bar{x}_{1}, \\
			\vdots \\
			x_{n-1}=\bar{x}_{n-1}.
	\end{cases}
\end{equation}\\
Cutting this straight line with the hyperplane $\Pi$ we get the point $(\bar{x}_1,\ldots,\bar{x}_{n-1},H_w),$ which gives the enunciated result.
A similar argument proves point b), just by considering in this case the straight line parallel to the $x_n$ axis passing through the barycenter $GM_w=\sum\limits_{i=1}^{n}w_iB_i=(w_1,\ldots,w_{n-1},0),$
and verifying that its intersection point with the hyperplane $\Pi$ is just the weighted arithmetic mean $M_w=\sum\limits_{i=1}^{n} w_i a_i.$
\end{proof}

\begin{theorem} \label{teo:NDcase3}
Let us consider the hyperplane $\Pi$ which passes through the points $P_i, i=1,\ldots,n,$ $n\geq 2,$ given by the equation
\begin{equation}\label{eq:Plano_Pi_3}
	\Pi\equiv \; x_n=a_{n} + \sum_{i=1}^{n}{x_{i}(a_{i}-a_{n})}.
\end{equation}
Let us also consider the paraboloid given by $V_{n}$ which passes through the points $B_i, i=1,\ldots,n-1$ and $P_n$ given by the equation
\begin{equation} \label{eqParab3-1-ND}
V_n \equiv  \; x_n=a_n+\sum_{i=1}^{n-1}(c_i x_{i}^2 -(c_i+a_n)x_i),
\end{equation}
where the coefficients $c_i$ are given by
\begin{eqnarray} \label{eqParab3-1-coef-ND}
c_{i}=\dfrac{\frac{H_{w}}{n}+(\sum\limits_{j=1}^{n-1}{\bar{x}_j}-1)a_n}{(n-1)\bar{x}_{i}(\bar{x}_{i}-1)}, \qquad i=1,\ldots,n-1,
\end{eqnarray}
and the paraboloids $V_i, i=1,\ldots,n-1$ given by the equations
\begin{eqnarray} \label{eqParab3-ND}
V_i &\equiv&	x_{n}=b_{i}{x_{i}}^{2} + (a_{i} - b_{i})x_{i},
\end{eqnarray}
which pass through $B_1,\ldots,B_{i-1},P_i,B_{i+1},\ldots, B_n$ respectively, where the coefficients $b_i$ are given by
\begin{equation}\label{eqParab3-coef-ND}
	b_{i}=\dfrac{H_{w}}{\bar{x}_{i}(\bar{x}_{i}-1)} (\frac{1}{n}-w_{i}), \qquad i=1,\ldots,n-1.
\end{equation}\\
Then, the system of equations formed by (\ref{eqParab3-1-ND}) and (\ref{eqParab3-ND}) has a unique solution $(\bar{x}_{1},\ldots,\bar{x}_{n})$ given by
\begin{equation}\label{eqnHwT4:Punto_xH3D}
\bar{x}_{i}=w_{i} \dfrac{H_{w}}{a_{i}}, \quad i=1,\ldots,n-1, \quad	\bar{x}_{n}=\frac{H_{w}}{n}.
\end{equation}
Moreover, the following two affirmations are true:
\begin{itemize}
\item[a)] The height of the prism through the point $(\bar{x}_1,\ldots,\bar{x}_{n-1},0)$ coincides with the weighted harmonic mean $H_w$ of $a_i,i=1,\ldots,n,$ that is, the point
 $(\bar{x}_1,\ldots,\bar{x}_{n-1},H_w)$ belongs to the hyperplane $\Pi.$
\item[b)] The height of the prism through the barycenter of the
base $GM_w=\sum\limits_{i=1}^{n}w_iB_i$ coincides with the weighted arithmetic mean.
\end{itemize}
\end{theorem}

\begin{proof}

It is trivial to see that the proposed solution satisfies (\ref{eqParab3-1-ND}) and (\ref{eqParab3-ND}).
Let us prove that the solution is unique. Let us suppose that there exists another solution
$x^{\prime}=(x^{\prime}_{1}, \cdots, x^{\prime}_{n})$ with $x^{\prime}\neq \bar{x}.$ Then, denoting $z_{i}= \bar{x}_{i}-x^{\prime}_{i}, i=1,\cdots,n,$ the system of equations formed by
(\ref{eqParab3-1-ND}) and (\ref{eqParab3-ND}) can be written as	
\begin{subequations}\label{eqnHwT1:T3_z}
	\begin{align}
			\sum\limits_{i=1}^{n-1}\left[{c_{i}z_{i}(\bar{x}_{i}+x^{\prime}_{i}) - (c_{i}+a_{n})z_{i}}\right] - z_{n} &= 0, \label{eqnHwT1:T3_z_1} \\
			z_{n}-b_{i}z_{i}(\bar{x}_{i}+x^{\prime}_{i}) - (a_{i} - b_{i})z_{i} &= 0,  \quad i=1,\cdots,n-1, \label{eqnHwT1:T3_z_2}
	\end{align}
\end{subequations}
what amounts to a homogeneous linear system of $n$ equations with $n$ unknowns. If we show that this system has only the trivial solution $z=0,$ then we would have proven that
$\bar{x}=x^{\prime},$ what is a contradiction with the starting supposition. Therefore, $\bar{x}$ would be the unique solution. Let us then prove that system (\ref{eqnHwT1:T3_z}) has
$z=0$ as the unique solution. By reductio ad absurdum, let us suppose that the system has infinite solutions, that is,
$z= \sum\limits_{k=1}^{s}{\lambda_{k}v^{k}}$, where $s=n-r,$ being $r$ the rank of the coefficient matrix of the linear system, $\lambda_{k} \in \mathbb{R},$ and $v^k, \ k=1,\ldots,s,$ represent
a base of the Kernel of the associated linear map. Let us consider the univariate set of solutions $z=\lambda_1 v^1,$ $\lambda_1 \in \mathbb{R}.$ By the sake of simplicity, we will drop the superindex and we will
write $z=\lambda v.$ Thus, we obtain
\begin{equation*}
x^{\prime}=\bar{x}-z= \bar{x}- \lambda v,
\end{equation*}
whose coordinates are given by\\
\begin{equation} \label{xprima}
x^{\prime}_{i}=\bar{x}_{i}-z_{i}= \bar{x}_{i}-\lambda v_{i}.
\end{equation}
Plugging (\ref{xprima}) into (\ref{eqnHwT1:T3_z_2}) we get\\
\begin{equation} \label{eq:lambda1-3}
\lambda v_{n}=b_{i} \lambda v_{i}(2 \bar{x}_{i}-\lambda v_{i}) + (a_{i} - b_{i})\lambda v_{i}, \quad \forall i=1,\cdots,n-1,
\end{equation}
and simplifying expression (\ref{eq:lambda1-3}) we obtain
\begin{equation}\label{eq:lambda2-3}
	-\lambda b_{i} v_{i}^{2}+2 b_{i} v_{i} \bar{x}_{i}+(a_{i} - b_{i})v_{i} - v_{n}=0, \quad \forall i=1,\cdots,n-1, \quad \lambda \neq 0.
\end{equation}
Now, particularizing expression (\ref{eq:lambda2-3}) for two different values of $\lambda,$  $\lambda_{1} \neq \lambda_{2},$ and subtracting both expressions, we reach to
\begin{equation}\label{eq:lambda3-3}
	(\lambda_{1} - \lambda_{2}) b_{i} v_{i}^{2} = 0, \quad \forall i=1, \cdots, n-1.
\end{equation}

Before continuing with the main proof, we need to prove the following statement
\begin{itemize}
\item[s1)] $sign(c_{i})=sign(c_j)$  $\forall i,j \ \in \{1,\ldots, n-1\},$
\end{itemize}
where $sign(.)$ denotes the sign function
\begin{equation*}
sign(x):=\left \{\begin{array}{ll}
1 & x>0,\\
-1 & x<0,\\
0 & x=0.
\end{array} \right.
\end{equation*}
Statement s1) is proven just by isolating the term $H_w$ in equation $c_{i}=0,$ that is
\begin{eqnarray}\label{eq:Hw}
c_{i}=0 &\Leftrightarrow& \frac{H_{w}}{n}=a_n(1-\sum_{j=1}^{n-1}\bar{x}_j)=a_n-H_w(a_n \sum_{j=1}^{n-1} \frac{w_j}{a_j})\\ \notag
&\Leftrightarrow& H_w=\dfrac{a_{n}}{\frac{1}{n}+a_n\sum\limits_{j=1}^{n-1}\frac{w_{j}}{a_j}}
=\dfrac{\prod\limits_{k=1}^{n}a_{k}}{\frac{1}{n}a_1\ldots a_{n-1}+\sum\limits_{j=1}^{n-1}w_{j}\prod\limits_{\substack{k=1 \\k \neq j}}^{n}{a_{k}}}.
\end{eqnarray}
Comparing expression (\ref{eq:Hw}) with the expression of $H_w$ in Definition \ref{Mn}, we get that $c_{i}=0  \Leftrightarrow w_n=\frac{1}{n},$
$c_{i}>0  \Leftrightarrow w_n<\frac{1}{n},$ and $c_{i}<0  \Leftrightarrow w_n>\frac{1}{n}.$ \\
We are ready to continue with the main proof.
Since $v\neq 0,$ the set of indices $I$ such that $v_i \neq 0, i \in I,$ is not empty. Let us suppose that $b_i =0, \forall i \in I.$
From (\ref{eqnHwT1:T3_z}), we get
\begin{eqnarray} \label{eq:s3-1}
v_n&=&a_{i}v_{i},\\
v_n&=&\sum_{i \in I}(-\lambda c_{i}v_{i}+2c_{i}\bar{x}_{i}-(a_n+c_{i}))v_{i}.\label{eq:s3-2}
\end{eqnarray}
Plugging (\ref{eq:s3-1}) into (\ref{eq:s3-2}) and using that $v_{i}\neq0,$ and in turn $v_n \neq 0,$ we get that
\begin{equation} \label{eq:s3-3}
\sum_{i \in I}(-\lambda c_{i}\frac{v_{i}}{a_i}+2c_{i}\frac{\bar{x}_{i}}{a_i}-\frac{a_n+c_{i}}{a_i})-1=0, \ \forall \lambda \in \mathbb{R}.
\end{equation}
From equation (\ref{eq:s3-3}), since it is true for all value of $\lambda,$ taking two different values $\bar{\lambda}_1, \bar{\lambda}_2$ we get that
\begin{eqnarray} \label{eq:s3-4}
-\bar{\lambda}_1v_n(\sum_{i \in I} \frac{c_{i}}{a_i^2})+\sum_{i \in I}(2c_{i}\frac{\bar{x}_{i}}{a_i}-\frac{a_n+c_{i}}{a_i})-1=0,\\
-\bar{\lambda}_2v_n(\sum_{i \in I} \frac{c_{i}}{a_i^2})+\sum_{i \in I}(2c_{i}\frac{\bar{x}_{i}}{a_i}-\frac{a_n+c_{i}}{a_i})-1=0, \label{eq:s3-5}
\end{eqnarray}
and subtracting (\ref{eq:s3-4}) and (\ref{eq:s3-5}) we get
\begin{equation} \label{eq:s3-6}
\sum_{i \in I} \frac{c_{i}}{a_i^2}=0.
\end{equation}
Taking into account statement s1), since all $c_i$ have the same sign, it must be $c_i=0, i \in I.$ In turn, by using (\ref{eq:s3-2}), this fact implies
\begin{equation*}
-\sum_{i \in I} \frac{a_nv_n}{a_i}=v_n \Rightarrow v_n(1+\sum_{i \in I} \frac{a_n}{a_i})=0,
\end{equation*}
what is not viable as $v_n\neq 0,$ and we get a contradiction. Therefore, $\exists i \in I,$ such that $b_i\neq 0.$
From (\ref{eq:lambda3-3}), this means that
$\; \lambda_{1} = \lambda_{2}$ what gives again a contradiction, this time with the initial supposition. Thus, $z=0$ is the unique solution of the homogeneous linear system and $\bar{x}$ is the unique solution of
 the system given by (\ref{eqParab3-1-ND}) and (\ref{eqParab3-ND}).\\

In order to prove now point a) of the theorem, we consider the straight line parallel to the $x_{n}$ axis passing through  $(\bar{x}_1,\ldots,\bar{x}_{n-1},0)$, that is
\begin{equation}\label{eqnAw:Recta_r-3}
	r_{w}\equiv
	\begin{cases}
			x_{1}= \bar{x}_{1}, \\
			\vdots \\
			x_{n-1}=\bar{x}_{n-1}.
	\end{cases}
\end{equation}\\
Cutting this straight line with the hyperplane $\Pi$ we get the point $(\bar{x}_1,\ldots,\bar{x}_{n-1},H_w),$ which gives the enunciated result.
A similar argument proves point b), just by considering in this case the straight line parallel to the $x_n$ axis passing through the barycenter $GM_w=\sum\limits_{i=1}^{n}w_iB_i=(w_1,\ldots,w_{n-1},0),$
and verifying that its intersection point with the hyperplane $\Pi$ is just the weighted arithmetic mean $M_w=\sum\limits_{i=1}^{n} w_i a_i.$
\end{proof}

\begin{remark} \label{remarkW}
In the non-weighted case, that is, when all $w_i=\frac{1}{n},i=1,\ldots,n,$ all the paraboloids degenerate in diagonal hyperplanes.
\end{remark}

A simpler representation using only hyperplanes is also possible for the general case of dealing with the weighted harmonic mean, as it comes out directly from Remark \ref{remarkW} and from the observation
\begin{equation} \label{wnw}
H_{\frac{1}{n}}(\frac{a_1}{w_1},\ldots,\frac{a_n}{w_n})=n H_w(a_1,\ldots,a_n),
\end{equation}
where $H_{\frac{1}{n}}$ stands for the harmonic mean with uniform weights $w_i=\frac{1}{n},i=1,\ldots,n.$

More precisely, using the previous notations and defining also
\begin{equation*}\label{eqnAH:PrismaHat}
	\begin{cases}
            a_i^*=\frac{a_i}{w_i},\\
			P_{i}^* \equiv (0, \cdots, 0,\underset{i}{1},0,\cdots,0, a_{i}^*), \quad i=1,\cdots,n-1,\\
			P_{n}^*= (0, \cdots, 0, a_{n}^*),\\
            H_w=H_w(a_1,\ldots,a_n), w=(w_1,\ldots,w_n),\\
            H_{\frac{1}{n}}^*=H_{w}(a_1^*,\ldots,a_n^*), w=(\frac{1}{n},\ldots,\frac{1}{n}),\\
            M_w=M_w(a_1,\ldots,a_n), w=(w_1,\ldots,w_n),\\
            M_{\frac{1}{n}}^*=M_{w}(a_1^*,\ldots,a_n^*), w=(\frac{1}{n},\ldots,\frac{1}{n}),
	\end{cases}
\end{equation*}
we can give the following corollary.

\begin{corollary} \label{Coro:NDcase}
Let us consider the hyperplanes $\Pi,$ $\Pi^*$ which pass through the points $P_i,$ and $P_i^*, i=1,\ldots,n,$ respectively, $n\geq 2.$ They are given by the equations
\begin{eqnarray}\label{eq:Plano_Pi_H1}
	\Pi &\equiv& \; x_n=a_{n} + \sum_{i=1}^{n}{x_{i}(a_{i}-a_{n})},\\
    \Pi^* &\equiv& \; x_n=\frac{a_{n}}{w_n} + \sum_{i=1}^{n}{x_{i}(\frac{a_{i}}{w_i}-\frac{a_{n}}{w_n})}.
\end{eqnarray}
Let us also consider the hyperplane $V_{n}^*$ which passes through the points $B_i, i=1,\ldots,n-1$ and $P_n^*$ given by the equation
\begin{equation} \label{eqPlaneHn-ND}
V_n^* \equiv  \; \sum_{i=1}^{n-1}{x_{i}} + w_n \dfrac{x_{n}}{a_{n}}=1.
\end{equation}
and the hyperplanes $V_i^*, i=1,\ldots,n-1$ given by the equations
\begin{eqnarray} \label{eqPlaneH-ND}
V_i^* &\equiv&	x_{n}=\frac{a_{i}}{w_i} x_{i}, \quad \textrm{which pass through} \ B_1,\ldots,B_{i-1},P_i^*,B_{i+1},\ldots, B_n.
\end{eqnarray}
Then, the system of equations formed by (\ref{eqPlaneHn-ND}) and (\ref{eqPlaneH-ND}) has a unique solution $(\bar{x}_{1},\ldots,\bar{x}_{n})$ given by
\begin{equation}\label{eqnHwTH:Punto_xH3D}
\bar{x}_{i}=w_{i} \dfrac{H_{w}}{a_{i}}, \quad i=1,\ldots,n-1, \quad	\bar{x}_{n}=\frac{H_{\frac{1}{n}}^*}{n}=H_w.
\end{equation}
Moreover, the following two affirmations are true:
\begin{itemize}
\item[a)] The height of the prism $P^*$ (with vertices $P_i^*$) through the point $(\bar{x}_1,\ldots,\bar{x}_{n-1},0)$ coincides with the harmonic mean of $a_i^*,i=1,\ldots,n,$ that is,
$H_{\frac{1}{n}}^*,$
and the height of the prism $P$ (with vertices $P_i$) through the same point is $\frac{1}{n}H_{\frac{1}{n}}^*=H_w,$ that is, the point
 $(\bar{x}_1,\ldots,\bar{x}_{n-1},H_w)$ belongs to the hyperplane $\Pi.$
\item[b)] The height of the prism $P^*$ through the barycenter of the
triangular base $GM_{\frac{1}{n}}=\sum\limits_{i=1}^{n}\frac{1}{n}B_i$ coincides with the arithmetic mean $M_{\frac{1}{n}}^*,$ and the height of the prism $P$ through the weighted barycenter of the
triangular base $GM_w=\sum\limits_{i=1}^{n}w_iB_i$ coincides with the weighted arithmetic mean $M_w.$
\end{itemize}
\end{corollary}

\begin{proof}
It is trivially derived either from Theorem \ref{teo:NDcase1} or from Theorem \ref{teo:NDcase3} just by applying the relation (\ref{wnw}) and observing the Remark \ref{remarkW}.
\end{proof}

In Figure \ref{fig:PlanesH} we see the representation of Corollary \ref{Coro:NDcase} in $2D$ and in $3D$ for a particular choice of the arguments $a_i$ and the weights. In the upper part
 we see the case of two arguments. We can appreciate the relation (\ref{wnw}) between the harmonic mean of the modified arguments $a_i^*$ and the weighted harmonic mean of the original arguments $a_i,$ which is placed at the half part of the height of the trapezoid at the abscissa where both means take place. However, no clear relation is observed between the arithmetic mean of the modified values $M_{\frac{1}{n}}^*$ and the weighted arithmetic mean $M_w$ of the original ones. The same appreciation runs for the case of three arguments, where the weighted harmonic mean of the original arguments locates at the third part of the height of
 the prism through the corresponding abscissa.

\begin{figure}[!ht] \label{fig:PlanesH}
\centerline{\includegraphics[width=7cm]{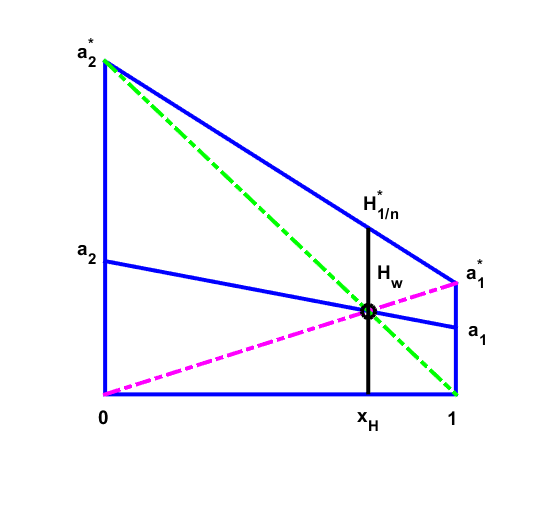}
\includegraphics[width=7cm]{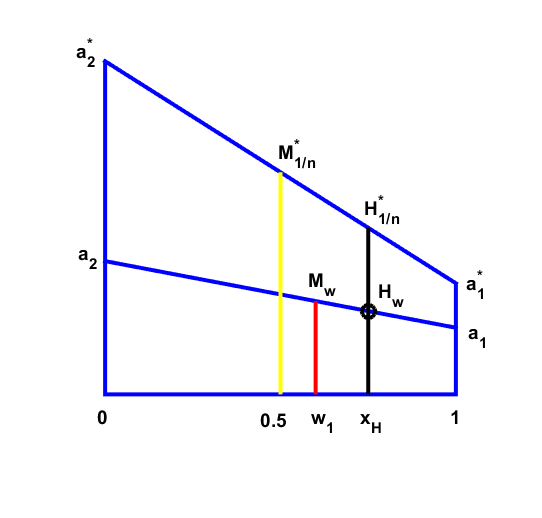}}
\centerline{\includegraphics[width=7cm]{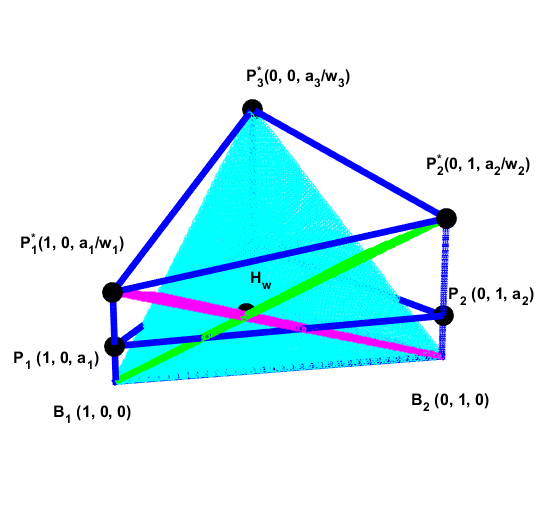}
\includegraphics[width=7cm]{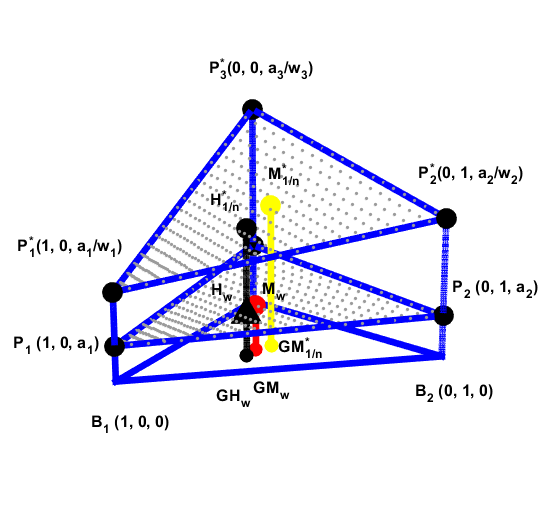}}
\caption{Representation of weighted harmonic mean according to Corollary \ref{Coro:NDcase}. Upper-left: Weighted harmonic mean of the two positive values $a_1=3,$ $a_2=6,$ with weights $w_1=0.6,$ $w_2=0.4.$
Upper-right: Comparison among $H_{\frac{1}{n}}^*,$ $H_w,$ $M_{\frac{1}{n}}^*,$ $M_w,$ in the case of two arguments.
Bottom-left: Weighted harmonic mean of the three positive values $a_1=6,$ $a_2=7,$ $a_3=10$ with weights $w_1=0.4,$ $w_2=0.3,$ $w_3=0.3.$
Bottom-right: Comparison among $H_{\frac{1}{n}}^*,$ $H_w,$ $M_{\frac{1}{n}}^*,$ $M_w,$ in the case of three arguments. In blue the harmonic mean, in red the weighted arithmetic mean of the original values,
in yellow the arithmetic mean of the modified values.}
\end{figure}

\section{Examples of application} \label{sec5}
In this section our main purpose is to point out how to use the simple theoretical results presented in previous sections to define a nonlinear reconstruction operator adapted to jump discontinuities. This application is just one possibility of use of the introduced concepts. It can be applied in many other contexts in order to define a nonlinear method from an already existing linear method, just by writing the necessary
expressions in terms of a weighted arithmetic mean of some quantities, which act somehow as smoothness indicators. Adaptation after the substitution of the weighted arithmetic mean for the corresponding 
weighted harmonic mean will take place if only a few of these quantities are affected by the presence of a potential discontinuity, and these affected quantities are not used 
in the rest of the expressions. These smoothness indicators should also satisfy similar hypothesis to those of Lemma \ref{ordenn} in smooth areas of an
hypothetical underlying function in order to maintain the approximation order of the original method.
The fact of substituting the arithmetic mean for a corresponding harmonic mean will allow the adaptation thanks to Lemma \ref{minimon}, since the large values,
due to the presence of a discontinuity, will be limited. \\
In real applications, it might be also necessary the application of a translation strategy as it is well explained in \cite{traslacion,Recons3DArxiv} in order to deal with weighted means of values which do not necessarily have the same sign.\\
Some examples of already existing methods that use these ideas with the harmonic mean of two values can be found in several applications. Let us mention:
\begin{itemize}
\item Point values reconstructions and the related field of subdivision and multiresolution schemes, see \cite{ADLT,OT3} and the references therein.
\item The field of image processing, to define nonlinear compression methods into the cell averages framework inside Harten's multiresolution, see \cite{ALRT}.
\item Also in the field of image processing for denoising purposes, see \cite{Den1}.
\item Generation of curves and surfaces, due to some remarkable properties of the harmonic mean in relation with the definition of convexity preserving reconstruction methods, see for example \cite{KD}.
\item In combination with spline reconstructions, see \cite{ACRT}.
\item In the solution of hyperbolic conservation laws, see \cite{Susa,Mar}.
\end{itemize}

Up to our knowledge, there is only one application using these ideas in $3D$ involving harmonic means of $3$ values, \cite{Recons3DArxiv}, and there is no other implementation of nonlinear algorithms based on 
this methodology in higher dimensions or involving harmonic means of more than $3$ values.

\section{Conclusions}\label{sec6}

In this article, we have presented two relevant properties of the harmonic mean that allow for new constructions of numerical methods, such as nonlinear reconstruction operators, subdivision and multiresolution
schemes, and solvers of hyperbolic conservation laws. These properties have been presented for any finite number of arguments, with the purpose of generating new algorithms in problems involving $N$-dimensional
spaces. We have given some geometrical representations of both the weighted harmonic mean and the weighted arithmetic mean, where the mentioned properties can be appreciated in an intuitive way. In the last part of the article, we offer a list of examples that illustrate the methodology for the two variables case, and we explain how to use these simple concepts to attain interesting and promising results in defining new methods for 
higher dimensions. In fact, we give also a reference of a particular new  reconstruction method for two dimensional functions which seems to avoid Gibbs effect at the time of retaining some approximation order close
to jump discontinuities.

\clearpage

\end{document}